\documentclass[12pt]{amsart}

\usepackage{mathrsfs}
\usepackage{hyperref}
\usepackage{courier}
\usepackage{epsfig}
\usepackage{graphicx}
\usepackage{amsmath}
\usepackage{amsthm}
\usepackage{mathtools}
\usepackage{amssymb}
\usepackage{verbatim}
\usepackage{booktabs}
\usepackage{algpseudocode}
\usepackage{amscd}
\usepackage{times}
\usepackage{subfigure} 
\usepackage{algorithm}
\usepackage{enumerate}
\usepackage{tikz}
\usepackage{epstopdf}
\usepackage{multirow}
\usepackage{bm}
\usepackage{natbib}
\usepackage{diagbox}
\usepackage{pgfplots}

\algrenewcommand\textproc{}

\newtheorem{thm}{Theorem}

\newtheorem{lem}{Lemma}
\newtheorem{prp}{Proposition}

\newtheorem{cor}{Corollary}
\newtheorem{remark}{Remark}

\def\bbA{\mathbb{A}}

\def\C{\mathcal{C}}

\def\Y{\mathcal{Y}}

\def\sS{\mathscr{S}}

\def\I{\mathfrak{I}}

\def\Q{\mathcal{Q}}

\def\Re{\mathbb{R}}

\def\sR{\mathscr{R}}

\def\U{\mathscr{U}}

\def\0{\boldsymbol{0}}
\def\1{\boldsymbol{1}}

\DeclareMathOperator*{\im}{im}

\DeclareMathOperator*{\codim}{codim}

\DeclareMathOperator*{\Span}{Span}

\DeclareMathOperator*{\id}{id}

\DeclareMathOperator*{\height}{height}

\DeclareMathOperator*{\F}{Fl}
\DeclareMathOperator*{\Gr}{Gr}

\DeclareMathOperator*{\Spec}{Spec}

\DeclareMathOperator*{\rank}{rank}

\DeclarePairedDelimiter\floor{\lfloor}{\rfloor}

\begin{document}

\title{Determinantal conditions for homomorphic sensing}

\author{Manolis C. Tsakiris}
\address{School of Information Science and Technology, ShanghaiTech University, No.393 Huaxia Middle Road, Pudong Area, Shanghai, China}
\address{Dipartimento di Matematica,  Universit\`a di Genova, Via Dodecaneso 35, 16146 Genova, Italy}
\email{mtsakiris@shanghaitech.edu.cn}

\begin{abstract}
With $k$ an infinite field and $\tau_1,\tau_2$ endomorphisms of $k^m$, we provide a dimension bound on an open locus of a determinantal scheme, under which, for a general subspace $V \subseteq k^m$ of dimension $n \le m/2$, for $v_1,v_2 \in V$ we have $\tau_1(v_1)=\tau_2(v_2)$ only if $v_1=v_2$. Specializing to permutations composed by coordinate projections, we obtain an abstract proof of the theorem of \cite{unnikrishnan2018unlabeled}. 
\end{abstract}

\maketitle

\section{Introduction}

In a fascinating line of research in signal processing termed \emph{unlabeled sensing}, it has been recently established that uniquely recovering a signal from shuffled and subsampled measurements is possible as long as the number of measurements is at least twice the intrinsic dimension of the signal \citep{unnikrishnan2018unlabeled}, while the source generating the signal is sufficiently exciting. In abstract terms, this says that if $V$ is a general\footnote{The attribute \emph{general} is used in the algebraic geometry sense, to indicate that the claimed property is true for every $V$ on a dense open set of the Grassmannian.} $n$-dimensional linear subspace of $\Re^m$, for some $m \ge 2n$, $\pi_1,\pi_2$ permutations on the $m$ coordinates of $\Re^m$ and $\rho_1,\rho_2$ coordinate projections viewed as endomorphisms, then $\rho_1 \pi_1(v_1) = \rho_2 \pi_2(v_2)$ implies $v_1=v_2$ whenever $v_1,v_2 \in V$, providing that each $\rho_i$ preserves at least $2n$ coordinates.
A similar phenomenon has been identified in real phase retrieval \citep{lv2018real, han2018recovery}. In both cases the proofs involve lengthy combinatorial arguments which show that certain determinants do not vanish. In this paper we provide an abstract justification for this phenomenon, that may very well go under the name \emph{homomorphic sensing}.  

Let $k$ be an infinite field and $\tau_1,\tau_2$ endomorphisms of $k^m$. Let $\rho$ be a linear projection onto $\im(\tau_2)$, that is $\rho$ is an idempotent endomorphism of $k^m$ with $\im(\rho) = \im(\tau_2)$. Let $R,T_1,T_2 \in k^{m \times m}$ be matrix representations of $\rho,\tau_1,\tau_2$ on the canonical basis of $k^m$. Let $k[x]=k[x_1,\dots,x_m]$ be a polynomial ring and $I_{\rho \tau_1, \tau_2}$ the ideal generated by all $2\times 2$ determinants of the $m \times 2$ matrix $[R T_1x \, \, \, T_2 x]$ with $x=x_1,\dots,x_n$ arranged as column vector. Consider the closed subscheme $Y_{\rho \tau_1,\tau_2} = \Spec (k[x] / I_{\rho \tau_1, \tau_2})$ of $\bbA_k^m=\Spec(k[x])$. Its $k$-valued points correspond to $w$'s in $k^m$ with $\rho \tau_1(w), \tau_2(w)$ are linearly dependent. With $V \subseteq k^m$ a $k$-subspace $\overline{V} = \Spec(k[x]/I_V)$ is the closed subscheme of $\bbA_k^m$ corresponding to $V$, where $I_V$ is the vanishing ideal of $V$. The key object is the locally closed subscheme $$U_{\rho \tau_1,\tau_2}  =  Y_{\rho \tau_1,\tau_2} \setminus \overline{\ker(\rho \tau_1 - \tau_2)}\cup \overline{\ker(\rho \tau_1)} \cup \overline{\ker(\tau_2)}$$

Let $\Gr(n,m)$ be the Grassmannian of $n$-dimensional $k$-subspaces of $k^m$, identified by the image of the Pl\"ucker embedding with an irreducible projective variety. Our main result is:

\begin{thm} \label{thm:tau1-tau2}
For $n\le m/2$ suppose $\dim U_{\rho \tau_1, \tau_2} \le m-n,  \dim_k\im(\tau_2)\ge 2n,  \dim_k\im(\tau_1)\ge n$. Then  there is an open dense set $\U \subseteq \Gr(n,m)$ such that for $V \in \U$ and $v_1,v_2 \in V$ we have $\tau_1(v_1) = \tau_2(v_2)$ only if $v_1=v_2$. 
\end{thm}

By a coordinate projection $\rho$ we mean an endomorphism of $k^m$ which preserves the values of $\rank(\rho)$ coordinates and sets the rest to zero. 

\begin{thm} \label{thm:permutations}
Let $\pi_1,\pi_2$ be permutations on the $m$ coordinates of $k^m$ and $\rho_1,\rho_2$ coordinate projections. Then $\dim U_{\rho_2 \rho_1 \pi_1, \rho_2 \pi_2} \le m-\floor*{\rank(\rho_2)/2}$. 
\end{thm} 

Using Theorems \ref{thm:tau1-tau2}-\ref{thm:permutations} we obtain a generalization of the main theorem of \cite{unnikrishnan2018unlabeled}. The generalization consists in allowing one of the projections to preserve at least $n$ coordinates (and not $2n$ for both projections) as well as considering sign changes. We call $\rho: k^m \rightarrow k^m$ a signed coordinate projection, if it is the composition of a coordinate projection with a map represented by a diagonal matrix with $\pm 1$ on the diagonal. 

\begin{cor} \label{cor:signed-permutations}
Let $\mathscr{P}_m$ be the group of permutations on the $m$ coordinates of $k^m$, and $\sR_n, \sR_{2n}, \sS_n, \, \sS_{2n}$ the set of all coordinate projections ($\sR_n, \, \sR_{2n}$) and signed coordinate projections ($\sS_n, \, \sS_{2n}$) of $k^m$, which preserve at least $n$ and $2n$ coordinates respectively, for some $n \le m/2$. Then the following is true for a general $n$-dimensional subspace $V$: if $\rho_1 \pi_1(v_1) = \rho_2 \pi_2 (v_2)$ for $v_1,v_2 \in V$ with $\rho_1 \in \sS_n, \rho_2 \in \sS_{2n}, \pi_1, \pi_2 \in \mathscr{P}_m$, then $v_1 = v_2$ or $v_1 = -v_2$. Moreover, if $\rho_1 \in \sR_n$ and $\rho_2 \in \sR_{2n}$, then $v_1 = v_2$. 
\end{cor}

Restricting our attention to general points, a much simpler argument gives much weaker conditions than in Theorem \ref{thm:tau1-tau2}: 
\begin{prp} \label{prp:generalHS}
Suppose $\tau_1,\tau_2$ have rank at least $n+1$ and are not scalar multiples of each other. Then for a general $n$-dimensional linear subspace $V$ of $k^m$ and $v$ a general point in $V$, we have $\tau_1(v)=\tau_2(v')$ with $v' \in V$ only if $v'=v$. 
\end{prp}

The results of the present paper substantiate the technical part of the expository paper of \cite{tsakiris2019homomorphic}. The reader may also find there an application of these notions to image registration. In a short letter \cite{dokmanic2019permutations} studied independently the same problem for $\tau_1,\tau_2$ automorphisms with $\tau_1^{-1}\tau_2$ diagonalizable. I was inspired to work on this problem after Prof. Aldo Conca pointed to me, during his visit in Shanghai in September 2018, eigenspace conditions in the form of Proposition \ref{prp:tau} for diagonalizable endomorphisms. I thank Liangzu Peng for stimulating discussions and comments on the manuscript.

\section{Proof of Theorem \ref{thm:tau1-tau2}}

For a positive integer $s$ set $[s] = \{1,\dots,s\}$ and $[0]=0$. We first consider a special case where one of the endomorphisms is the identity $\id$. Let $\tau$ be the other endomorphism with $T \in k^{m \times m}$ its matrix representation on the canonical basis of $k^m$. Denote by $I_\tau$ the ideal 
of $k[x]$ generated by the $2 \times 2$ determinants of the $2 \times m$ matrix $[Tx \, \, \, x]$. The $k$-valued points of the closed subscheme $Y_\tau = \Spec (k[x]/I_\tau)$ form the union of the eigenspaces of the endomorphism $\tau$ corresponding to eigenvalues that lie in $k$. Set $U_\tau = Y_\tau \setminus \overline{\ker(\tau- \id)}$ the open subscheme of $Y_\tau$ with the locus associated to eigenvalue $1$ removed. We have: 

\begin{prp} \label{prp:tau}
Suppose that $\dim U_\tau \le m-n$ for some $n$ with $m \ge 2n$. Then there is a dense open set $\U \subset \Gr(n,m)$ such that for every $V \in \U$ and $v_1,v_2 \in V$ we have $\tau(v_1) = v_2$ only if $v_1=v_2$.
\end{prp}
\begin{proof}
See \S \ref{subsection:Proof-prp:tau}. 
\end{proof}

Denote by $k[x]_1$ the $k$-vector space of degree-$1$ homogeneous polynomials in $k[x]$. Write $\overline{\ker(\rho \tau_1-\tau_2)} = \Spec(k[x] / J)$ where $J$ is generated by linear forms $p_\alpha \in k[x]_1, \, \alpha \in [\codim\ker(\rho \tau_1-\tau_2)]$. Similarly, let $q_\beta$'s and $r_\gamma$'s be linear forms generating the vanishing ideals of $\ker(\rho \tau_1)$ and $\ker(\tau_2)$ respectively. Set $h_{\alpha \beta \gamma} = p_\alpha q_\beta r_\gamma$. Then 
$$U_{\rho \tau_1, \tau_2} = \bigcup_{\alpha, \beta, \gamma} \Spec  \big(k[x] / I_{\rho \tau_1,\tau_2}\big)_{h_{\alpha \beta \gamma}} $$
where $\big(k[x] / I_{\rho \tau_1,\tau_2}\big)_{h_{\alpha \beta \gamma}}$ is the localization of $k[x] / I_{\rho \tau_1,\tau_2}$ at the multiplicatively closed set $\{1,h_{\alpha \beta \gamma}, h_{\alpha \beta \gamma}^2,\dots\}$.

Set $\ell = \dim_k \im(\tau_2)$. There is a dense open set $\U_1 \subseteq \Gr(\ell,m)$ such that $H \cap \ker(\tau_2) = 0$ for every $H \in \U_1$. For any such $H$ we have that $\tau_2|_H$ establishes an isomorphism between $H$ and $\im(\tau_2)$. Let $(\tau_2|_{H})^{-1}: \im(\tau_2) \rightarrow H$ be the inverse map. Consider the endomorphism of $H$ given by $\tau_H=(\tau_2|_{H})^{-1} \rho \tau_1|_{H}$. Fixing a basis $B_{\im(\tau_2)} \in k^{m \times \ell}$ of $\im(\tau_2)$ we let $R'$ be the $k^{\ell \times m}$ matrix that sends a vector $\xi \in k^m$ to the coefficients of the representation of $\rho(\xi)$ on the basis $B_{\im(\tau_2)}$ and note that $R'R=R'$. Fix a basis $B_H \in k^{m \times \ell}$ of $H$. Then $\tau_2|_H$ is represented by the invertible matrix $T_{2,H}=R'T_2B_H \in k^{\ell \times \ell}$ and $\tau_H$ by $T_H= T_{2,H}^{-1} R' T_1 B_H \in k^{\ell \times \ell}$. 

Let $k[z] = k[z_1,\dots,z_\ell]$ be a polynomial ring of dimension $\ell$ and consider the surjective ring homomorphism $\psi: k[x] \rightarrow k[z]$ that takes $x$ to $B_H z$. The kernel of $\psi$ is the vanishing ideal $I_H$ of $H$, so that $\psi$ induces a ring isomorphism $k[x]/I_H \cong k[z]$. The further ring isomorphism $k[x]/I_H+(p_\alpha)_\alpha \cong k[z] / (\psi(p_\alpha))_\alpha$ corresponds geometrically to the identification $\ker(\rho \tau_1 - \tau_2) \cap H \cong \ker(\tau_H - \id)$. That is, the $\psi(p_\alpha)$'s generate the vanishing ideal of $\ker(\tau_H - \id)$. Similarly, the $\psi(q_\beta)$'s generate the vanishing ideal of $\ker(\tau_H)$, while the $\psi(r_\gamma)$'s generate the irrelevant ideal $(z_1,\dots,z_\ell)$. Now define $I_{\tau_H}$ to be the ideal of $k[z]$ generated by all $2 \times 2$ determinants of the $\ell \times 2$ matrix $[T_H z \, \, \, z]$. We have:

\begin{lem} \label{lem:2-minors}
$I_{\tau_H}  = \psi(I_{\rho \tau_1,\tau_2})$.
\end{lem}
\begin{proof}
Since $\tau_2 (\tau_2|_{H})^{-1} \rho = \rho$ we have $\psi([R T_1 x \, \, \, T_2 x]) = T_2 B_H [T_H z \, \, \, z]$. Recall that if $C$ is a $2 \times \ell$ row-submatrix of $T_2 B_H$ then $\det(C [T_H z \, \, \, z]) = \sum_{\mathcal{J} \subset [\ell], \#\mathcal{J}=2} \det(C_{\mathcal{J}}) \det({}_{\mathcal{J}}[T_H z \, \, \, z])$, where $C_{\mathcal{J}}$ and ${}_{\mathcal{J}}[T_H z \, \, \, z]$ denote column and row $2 \times 2$ submatrices repsectively, indexed by $\mathcal{J}$. This shows that the ideal of $2 \times 2$ determinants of $\psi([R T_1 x \, \, \, T_2 x])$ is contained in the ideal of $2 \times 2$ determinants of $[T_H z \, \, \, z]$. For the reverse inclusion, note that $T_2 B_H$ has rank $\ell$ and so there is an invertible row-submatrix $A$ of $T_2 B_H$ of size $\ell \times \ell$. It is enough to prove that the ideal of $2 \times 2$ determinants of $A[T_H z \, \, \, z]$ coincides with that of $ [T_H z \, \, \, z]$. The matrix $A$ induces a $k$-automorphism $f: (k[z])^\ell \rightarrow (k[z])^\ell$ given by $u \mapsto Au$. This further induces a $k$-linear map of exterior powers $f^{(2)}: \wedge^2 (k[z])^\ell \rightarrow \wedge^2 (k[z])^\ell$ by taking $u \wedge v$ to $Au \wedge Av$. Note that $u \wedge v$ is the vector of $2 \times 2$ determinants of the matrix $[u \, \, v]$. Similarly $A^{-1}$ induces a $k$-linear map $g^{(2)}: \wedge^2 (k[z])^\ell \rightarrow \wedge^2 (k[z])^\ell$. Since $f^{(2)}, g^{(2)}$ are inverses, the vectors of $2 \times 2$ determinants of $A[T_H z \, \, \, z]$ and $ [T_H z \, \, \, z]$ can be obtained from each other via matrix multiplication over $k$, thus they generate the same ideal. 
\end{proof}

Lemma \ref{lem:2-minors} gives the ring isomorphism $\psi': k[x]/I_{\rho \tau_1,\tau_2}+I_H \cong k[z]/I_{\tau_H}$. Together with the definition of $U_{\tau_H}$ this gives
\begin{align}
U_{\tau_H} \setminus \overline{\ker(\tau_H)}&=\Spec(k[z]/I_{\tau_H})\setminus \overline{\ker(\tau_H-\id)} \cup \overline{\ker(\tau_H)} \nonumber \\
& \cong \bigcup_{\alpha,\beta} \Spec  \big(k[x] / I_{\rho \tau_1,\tau_2}+I_H\big)_{p_\alpha q_\beta}  \nonumber 
\end{align} 
 
Let $\Gr(c,k[x]_1)$ be the Grassmannian of $k$-subspaces $W$ of $k[x]_1$ of dimension $c$. The following is a folklore fact in commutative algebra. 

\begin{lem} \label{lem:sections-algebraic}
Let $I$ be a homogeneous ideal of $k[x]$. Then there exists a dense open set $\U^* \subseteq \Gr(c,k[x]_1)$ such that 
$$\dim \big(k[x]/I+(W)\big) = \max\{ \dim \big(k[x]/I\big) - c, 0 \}$$ for every $W \in \U^*$, with $(W)$ the ideal generated by $W$. 
\end{lem}

Let $\mathscr{P}$ be the minimal set of homogeneous prime ideals of $k[x]$ such that $\sqrt{I_{\rho \tau_1,\tau_2}} = \bigcap_{P \in \mathscr{P}} P$. Let $\mathscr{P}_{\alpha,\beta}$ be the subset of $\mathscr{P}$ consisting of those $P$'s that do not contain $p_\alpha q_\beta$ for some $\alpha,\beta$. Each such $P$ corresponds to an irreducible component of $\Spec  \big(k[x] / I_{\rho \tau_1,\tau_2}\big)_{p_\alpha q_\beta}$. For $P \in \mathscr{P}_{\alpha,\beta}$ Lemma \ref{lem:sections-algebraic} with $c=m-\ell$ and $I=P$ gives a dense open set $\U_P^* \subseteq \Gr(m-\ell,k[x]_1)$ on which $\dim \big(k[x]/P+(W)\big) = \max\{ \dim \big(k[x]/P\big) - m+\ell, 0 \}$ for every $W \in \U_P^*$. Set $\U^*_{\alpha,\beta} = \bigcap_{P \in \mathscr{P}_{\alpha,\beta}} \U_P^*$. For every $W \in \U^*_{\alpha,\beta}$ we have $\dim  \big(k[x] / I_{\rho \tau_1,\tau_2} +(W) \big)_{p_\alpha q_\beta} = \max\big\{ \dim \big(k[x]/I_{\rho \tau_1,\tau_2}\big)_{p_\alpha q_\beta} - m+\ell, 0 \big\}$. By hypothesis $\dim U_{\rho \tau_1, \tau_2} \le m-n$ so $\dim \big(k[x]/I_{\rho \tau_1,\tau_2}\big)_{p_\alpha q_\beta} \le m-n$. With $\U^* = \bigcap_{\alpha,\beta} \U^*_{\alpha,\beta}$, for every $W \in \U^*$ we have that $\dim \big(k[x] / I_{\rho \tau_1,\tau_2} +(W) \big)_{p_\alpha q_\beta} \le \ell -n$ for every $\alpha,\beta$. Now under the isomorphism $\Gr(m-\ell,k[x]_1) \cong \Gr(\ell,m)$ the open set $\U^*$ gives an open set $\U_2 \subset \Gr(\ell,m)$ such that $H \in \U_2$ if and only if $I_H \in \U^*$. We conclude that $\dim U_{\tau_H} \setminus \overline{\ker(\tau_H)} \le \ell -n$ for every $H \in \U_2$.

The locus of $H$'s in $\U_1 \cap \U_2$ for which i) $\dim_k \ker(\tau_H)$ is minimal, ii) $\dim_k E_{\tau_H,1}=\ell - \rank (R'T_1B_H - R'T_2B_H)$ is minimal, iii) a unique basis $B_H$ exists with the top $\ell \times \ell$ block the identity matrix, is also open and non-empty; call it $\U_3$. For every $H \in \U_3$ the above mentioned unique representation $B_H$ of $H$ establishes a $k$-vector space isomorphism $H \cong k^\ell$ by sending the $j$th column of $B_H$ to the $j$th canonical vector of $k^\ell$. This further establishes an isomorphism of projective varieties $\gamma_H: \Gr(n,H) \xrightarrow{\sim} \Gr(n,\ell)$. By the definition of $\U_3$ $\dim_k \ker(\tau_H)$ is constant for every $H \in \U_3$, call that value $\alpha$. If $\alpha \le n$ then $\tau_H$ satisfies the hypothesis of Proposition \ref{prp:tau} for every $H \in \U_3$. Hence there is a dense open set $\U_H \subseteq \Gr(n,H)$ such that for every $V \in \U_H$ and $v_1,v_2 \in V$ we have $\tau_H(v_1) = v_2$ only if $v_1=v_2$. If on the other hand $\alpha > n$, it is easy to see that there is another dense open set that we also call $\U_H \subseteq \Gr(n,H)$, such that for every $V \in \U_H$ and $v_1,v_2 \in V$ the equality $\tau_H(v_1) = v_2$ implies $v_1=v_2=0$. 
 
We now show that the incidence correspondence $V \subset H$ with $H \in \U_4$ and $V \in \U_H$ contains a non-empty open set of the flag variety $\F(n,\ell,m)$, the latter defined as the closed subset of $\Gr(n,m) \times \Gr(\ell,m)$ cut out by the relation $V \in \Gr(n,H)$. Towards that end, it is enough to show that the equations that define $\U_H$ are polynomials in the Pl\"ucker coordinates of $V$ via $\gamma_H$ with rational coefficients in $B_H$. Denote by $k(B_H)$ the field of fractions of the polynomial ring $k[B_H]$ with the free entries of $B_H$ viewed as variables. The parametrization of $\U_H$ by $H$ depends on the two numbers $\alpha= \dim_k \ker(\tau_H)$ and $\beta = \dim_k E_{\tau_H,1}$. Both these dimensions are constant for every $H \in \U_3$ and there are three possibilities for the structure of $\U_H$ determined by the cases i) $\alpha \le n, \, \beta \le m-n$, ii) $\alpha \le n, \, \beta > m-n$, iii) $\alpha >n$. We only discuss i) and ii). For case i) the last part of the proof of Proposition \ref{prp:tau} shows that $\U_H$ is determined via $\gamma_H$ by the condition $\rank [T_H A \, \, \,  A] = 2n$, where $A \in k^{\ell \times n}$ is any basis of $\gamma_H(V)$. This amounts to the non-simultaneous vanishing of certain quadratic equations in the Pl\"ucker coordinates of $\gamma_H(V)$ with coefficients in $k(B_H)$. For case ii) we note that the number $\beta$ is equal to the $k(B_H)$-vector space dimension of the right nullspace of the matrix $T_H - I$, where $I$ is the identity matrix of size $\ell$. By Gauss-Jordan elimination over $k(B_H)$ we compute a $k(B_H)$-basis $\mathfrak{s}_1,\dots, \mathfrak{s}_{\beta} \in k(B_H)^\ell$ for that nullspace. We extend this sequence by adding vectors $s_1,\dots,s_{\ell -\beta} \in k^\ell$ such that the matrix $S  =[ \mathfrak{s}_1 \cdots \mathfrak{s}_\beta \, \, \,  s_1 \cdots s_{\ell - \beta} ] \in k(B_H)^{\ell \times \ell} $ is invertible over $k(B_H)$. The last part of the proof of Proposition \ref{prp:tau} shows that now $\U_H$ is determined via $\gamma_H$ as the $\gamma_H(V)$'s with basis $A \in k^{\ell \times n}$ for which $\det(S_{[n]}^{-1} A) \neq 0$, where $S_{[n]}^{-1}$ is the top $n \times m$ block of $S^{-1}$. This is a linear equation in the Pl\"ucker coordinates of $\gamma_H(V)$ with rational coefficients in $B_H$. 

We have a non-empty open set $\mathscr{O} \subset \F(n,\ell,m)$ such that for every $(V,H) \in \mathscr{O}$ we have that $V$ satisfies the property of interest: if $\tau_1(v_1) = \tau_2(v_2)$ for $v_1,v_2 \in V$ then $\tau_H(v_1) = v_2$ and thus necessarily $v_1=v_2$. The equations that define $\mathscr{O}$ also define a non-empty open subscheme $\overline{\mathscr{O}}$ of the flag scheme $\overline{\F}(n,\ell,m)$, where the overline notation indicates scheme structure. Now, since both $\overline{\F}(n,\ell,m)$ and $\overline{\Gr}(n,m)$ are irreducible, the image of $\overline{\mathscr{O}}$ under the canonical projection  $\overline{\F}(n,\ell,m) \rightarrow \overline{\Gr}(n,m)$ is dense. By Chevalley's theorem that image is constructible and thus it contains a non-empty open set $\overline{\U}_5$ whose $k$-valued points satisfy our property of interest. It remains to show how to get the open set $\U \subset \Gr(n,m)$ of the theorem. $\Gr(n,m), \overline{\Gr}(n,m)$ are locally isomorphic to the affine space of dimension $n(m-n)$. Let $\overline{\U}_6$ be the open set of $\overline{\Gr}(n,m)$ where some Pl\"ucker coordinate does not vanish. With $Y$ an $n \times (m-n)$ matrix of indeterminates, $\overline{\U}_6$ is isomorphic to $\mathbb{A}^{n(m-n)}=\Spec k[Y]$. The non-vanishing of the same Pl\"ucker coordinate in $\Gr(n,m)$ gives an open set $\U_7 \subset \Gr(n,m)$, which is isomorphic to $k^{n(m-n)}$. Replacing $\overline{\U}_5$ by its intersection with $\overline{\U}_6$, we may assume that it lies in $\mathbb{A}^{n(m-n)}$. As $\overline{\U}_5$ is covered by basic affine open sets, we may further assume that $\overline{\U}_5 = \Spec (k[Y])_p$ for some non-zero polynomial $p \in k[Y]$. Our open set $\U$ is the non-vanishing locus of $p$ in $\U_7$, which is non-empty by the infinity of $k$.

\subsection{Proof of Proposition \ref{prp:tau}} \label{subsection:Proof-prp:tau}

We recall some notions from linear algebra following \cite{Roman}. For simplicity we write $\tau v$ instead of $\tau(v)$. We say that a $k$-subspace $C$ of $k^m$ is $\tau$-cyclic if it admits a basis of the form $v,\tau v, \tau^2 v, \dots, \tau^{d-1}v$ for some $v \in k^m$ with $d=\dim_k C$. Let $y$ be a transcendental element over $k$. Then $k^m$ admits a $k[y]$-module structure under the action $p(y) \in k[y] \mapsto p(\tau) \in \operatorname{Hom}_k(k^m,k^m)$. Let $m_{\tau}(y)$ be the monic minimal polynomial of $\tau$ and let $m_{\tau}(y) = p_1^{\ell_1}(y) \cdots p_s^{\ell_s}(y)$ be its unique factorization into powers of irreducible polynomials $p_i(y) \in k[y]$. Then $k^m$ admits a primary cyclic decomposition as a $k[y]$-module into the direct sum of $\tau$-cyclic subspaces on which the minimal polynomial of $\tau$ is a power of one of the $p_i(y)$'s. Now $\tau$ admits an eigenvalue $\lambda \in k$ if and only if $y-\lambda$ divides $m_\tau(y)$, that is if and only if one of the $p_i(y)$'s is equal to $y-\lambda$. Let $C$ be a $\tau$-cyclic subspace as above in the primary decomposition with minimal polynomial of the form $(y-\lambda)^{e}$. Then $w_i=(\tau - \lambda)^{d-i}v, \, i \in [d]$ is a basis of $C$ with $\tau w_1 = \lambda w_1$ and $\tau w_i = \lambda w_i + w_{i-1}, \, i=2,\dots,d$. We call this basis a Jordan basis and the matrix representation of $\tau|_C$ on that basis is a Jordan block  
$$  \begin{bmatrix} 
\lambda & 1 & 0 & \cdots & 0 & 0 \\
0 & \lambda & 0 & \cdots & 0 & 0 \\
\vdots & \vdots & \vdots & \ddots & \vdots & \vdots \\
0 & 0 & 0 & \cdots & \lambda & 1 \\ 
0 & 0 & 0 & \cdots & 0 & \lambda 
\end{bmatrix} \in k ^{d \times d}$$
Thus the geometric multiplicity of the eigenvalue $\lambda \in k$ is the number of $\tau$-cyclic subspaces in the primary decomposition of $k^m$ for which $m_{\tau|_C}(y) = (y-\lambda)^d$ for some $d \ge 1$. Now $\tau$ induces an endomorphism of $\bar{k}^m$ in a natural way, which we also call $\tau$. With $\lambda_i, i \in [s]$ the eigenvalues of $\tau$ over $\bar{k}$ we have that $\bar{k}^m$ admits a decomposition $\bar{k}^m = \bigoplus_{t,i} C_{t,\lambda_i}$ into $\tau$-cyclic $\bar{k}$-subspaces with $C_{t, \lambda_i}$ corresponding to eigenvalue $\lambda_i$. That is each $\C_{t,\lambda_i}$ admits a Jordan basis $w_1,\dots,w_{d_{ti}}$ such that $\tau w_1 = \lambda_i w_1$ and $\tau w_j = \tau_i w_j + w_{j-1}, \, \, \, \forall j =2,\dots,d_{ti}$. We denote by $E_{\tau,\lambda}$ the eigenspace of $\tau$ associated to eigenvalue $\lambda$. We note that if $\lambda \in k$ then $\dim_k E_{\tau,\lambda} \cap k^m = \dim_{\bar{k}} E_{\tau,\lambda}$. Finally, with $K=k,\bar{k}$ we denote by $\Gr_K(n,m)$ the set of all $n$-dimensional $K$-subspaces of $K^m$.

We prove the proposition in several stages, starting with the boundary situation described in the next lemma.

\begin{lem} \label{lem:mu=n}
Suppose $m=2n$ and $\dim_{\bar{k}} E_{\tau,\lambda} =n$ for some $\lambda \in \bar{k}$. Then there exists a $V \in \Gr_{\bar{k}}(n,m)$ such that $\bar{k}^{m} = V \oplus \tau(V)$.   
\end{lem}
\begin{proof}
Let $\lambda_i, i \in [s]$ be the spectrum of $\tau$ over $\bar{k}$ and suppose that $\lambda_1=\lambda$ is the said eigenvalue. Then in the decomposition above of $\bar{k}^{2n}$ there are exactly $n$ subspaces $C_{t,\lambda_1}, t \in [n]$ associated to $\lambda_1$ each of them contributing a single eigenvector. With $w_{t,1},\dots,w_{t,d_t}$ a Jordan basis for $C_{t,\lambda_1}$ that eigenvector is $w_{t,1}$ and we set $v_t = w_{t,1}$ for $t \in [n]$. We produce $n$ linearly independent vectors $u_t, t \in [n]$ to be taken as a basis for the claimed subspace $V$, by summing pairwise the $v_t$'s with the remaining Jordan basis vectors across all $C_{t,\lambda_i}$'s in a manner prescribed below.

First, suppose that all $\C_{t,\lambda_1}$'s are $1$-dimensional. Then $\C_{1,\lambda_2}$ is a non-trivial subspace with Jordan basis say $w_1,\dots,w_d$, for some $d \ge 1$. We construct the first $d$ basis vectors $u_1,\dots,u_d$ for $V$ as $u_j = v_j + w_j, \, \, \, j \in [d]$.
A forward induction on the relations
\begin{align}
\tau u_1  = \lambda_1 v_1 + \lambda_2 w_1; \, \, \, \tau u_j  = \lambda_1 v_j + \lambda_2 w_j + w_{j-1}, \, \, \,  j=2,\dots,d, \nonumber 
\end{align} together with $\lambda_1 \neq \lambda_2$, gives 
$$\Span(u_1,\tau u_1,\dots,u_d,\tau u_d) = \left( \oplus_{t \in [d]} \C_{t,\lambda_1} \right)  \oplus \C_{1,\lambda_2}$$ If $d=n$ we are done, otherwise either $\C_{2,\lambda_2}$ or $\C_{1,\lambda_3}$ is a non-trivial subspace and we inductively repeat the argument above until all $\C_{t,\lambda_i}$'s are exhausted. 
     
Next, suppose that not all $C_{t,\lambda_1}$'s are $1$-dimensional. We may assume that there exists integer $0 \le r < n$ such that $\dim C_{t,\lambda_1} =1$ for every $t \le r$ and $\dim C_{t,\lambda_1} =d_j>1$ for every $t >r$. If $r=0$, then each $C_{t,\lambda_1}$ is necessarily $2$-dimensional and $\tau$ has only one eigenvalue $\lambda_1$. Letting $w_{1,t},w_{2,t}$ be the Jordan basis for $C_{t,\lambda_1}$, we define $u_t = w_{2,t}, \, \, \, \forall t \in [n]$. Clearly, 
$\Span(u_t, \tau u_t) = \Span(w_{1,t},w_{2,t})$, in which case $\Span\left(\{u_t,\tau u_t\}_{t \in [n]} \right) = \oplus_{t \in [n]} C_{t,\lambda_1} = \bar{k}^{2n}$. So suppose $1 \le r < n$. Let $w_1,\dots,w_{d_{r+1}}$ be a Jordan basis for $C_{r+1,\lambda_1}$. Since 
\begin{align}
2(n-r-1) \le \dim \oplus_{t=r+2}^{n} C_{t,\lambda_1} \le  \codim \oplus_{t=1}^{r+1} C_{t,\lambda_1} =2n-r-d_{r+1}, \nonumber
\end{align} we must have $$d_{r+1}-2 \le r$$ Recall that $w_1 = v_{r+1}$ and define $u_1 = v_{r+1}+w_{d_{r+1}}$ and $u_j=v_{j-1}+w_j$ for $j=2,\dots,d_{r+1}-1$. Noting that $\{w_j: j\in[d_{r+1}-1]\}=\{\tau u_j - \lambda u_j: j\in[d_{r+1}-1]\}$, we have $$\Span\left(\{u_j,\tau u_j\}_{j=1}^{d_{r+1}-1} \right) = \left(\oplus_{t=1}^{d_{r+1}-2} C_{t,\lambda_1}\right) \oplus  C_{r+1,\lambda_1}$$ If $r=n-1$, we have found a $(d_n-1)$-dimensional subspace $V':=\Span(u_j: \, j \in [d_{r+1}-1])$ such that $V'+\tau(V') = \oplus_{t=1}^n C_{t,\lambda_1}$. Otherwise if $r< n-1$, $C_{r+2,\lambda_1}$ is a nontrivial subspace of dimension $d_{r+2} \ge 2$, which must satisfy $r + d_{r+1} +d_{r+2} +2(n-r-2) \le 2n$ or 
$$d_{r+2}-2 \le r - (d_{r+1}-2)$$ Letting $w_1,\dots,w_{d_{r+2}}$ be a Jordan basis for $C_{r+2,\lambda_1}$ and recalling the convention $v_{r+2}=w_1$, we define $u_{d_{r+1}},\dots,u_{d_{r+1}+d_{r+2}-2}$ as 
\small
\begin{align}
u_{d_{r+1}} = v_{r+2}+w_{d_{r+2}},\, \, \, u_{d_{r+1}-1+j} = v_{d_{r+1}-3+j} + w_j, \, \, \, \forall  j=2,\dots, d_{r+2}-1 \nonumber
\end{align} \normalsize Then one verifies that
\small
\begin{align}
\Span\left(\{u_{d_{r+1}-1+j},\tau u_{d_{r+1}-1+j}\}_{j=1}^{d_{r+2}-1} \right) = \Big(\oplus_{t=1}^{d_{r+2}-2} C_{d_{r+1}-2+t,\lambda_1}\Big) \oplus C_{r+2,\lambda_1} \nonumber
\end{align} \normalsize and in particular 
\small
\begin{align}
 \Span\left(\{u_{j},\tau u_{j}\}_{j=1}^{d_{r+1}+d_{r+2}-2} \right) =  \Big(\oplus_{t=1}^{d_{r+1}+d_{r+2}-4} C_{t,\lambda_1}\Big) \oplus \Big(\oplus_{t \in [2]} C_{r+t,\lambda_1} \Big) \nonumber 
\end{align} \normalsize Continuing inductively like this we exhaust all $C_{t,\lambda_1}$'s that have dimension greater than $1$ and obtain $$V'=\Span\Big(\Big\{u_j: \, j=1,\dots,\sum_{j \in [n-r]} (d_{r+j}-1)\Big\} \Big)$$ 
$$V'+\tau(V')=\Big(\oplus_{t \in \left[\sum_{j=1}^{n-r}(d_{r+j}-2)\right]} C_{t,\lambda_1}\Big) \oplus \Big(\oplus_{t \in [n-r]}C_{r+t,\lambda_1} \Big)$$
 with $\sum_{j=1}^{n-r}(d_{r+j}-2) \le r$. If equality is achieved then $\dim V' = n$ and we can take $V=V'$; note that in that case $s=1$. Otherwise, $\dim \oplus_{t; i>1} C_{t,\lambda_i} = r - \sum_{j=1}^{n-r}(d_{r+j}-2)=:\alpha$ and this is precisely the number of $1$-dimensional $C_{t,\lambda_1}$'s that have not been used so far. Letting $\xi_1,\dots,\xi_{\alpha}$ be the union of all Jordan bases of all $C_{t,\lambda_i}$'s for $i>1$, we define the remaining $\alpha$ basis vectors of $V$ as 
$u_{n-\alpha+j} = v_{r-\alpha+j}+\xi_j, \, j \in [\alpha]$, and since 
\begin{align}
\Span\left(\{u_{n-\alpha+j},\tau (u_{n-\alpha+j})\}_{j=1}^{\alpha} \right) =  \Big(\oplus_{j=1}^{\alpha} C_{r-\alpha+j,\lambda_1}\Big) \oplus \Big(\oplus_{t; i>1} C_{t,\lambda_i}\Big) \nonumber
\end{align} the proof is complete.
\end{proof}

We now use Lemma \ref{lem:mu=n} to get a stronger statement for eigenspace dimensions less than or equal to half of the ambient dimension.
\begin{lem} \label{lem:mu<=n}
Suppose $m=2n$ and $\dim_{\bar{k}} E_{\tau,\lambda} \le n$ for every $\lambda \in \bar{k}$. Then there exists a $V \in \Gr_{\bar{k}}(n,m)$ such that $\bar{k}^{m} = V \oplus \tau(V)$.    
\end{lem}
\begin{proof}

Let $\lambda_i, i \in [s]$ be the eigenvalues of $\tau$ over $\bar{k}$ and proceed by induction on $n$. For $n=1$ we have $s\le 2$ and $\dim E_{\tau,\lambda_i} = 1$, whence the claim follows from Lemma \ref{lem:mu=n}. So let $n>1$. If $\dim_{\bar{k}} E_{\tau,\lambda_i} = n$ for some $i$, then we are done by Lemma \ref{lem:mu=n}. Hence suppose throughout that $\dim_{\bar{k}} E_{\tau,\lambda_i} < n, \, \forall i \in [s]$. Since the induction hypothesis applied on any $2(n-1)$-dimensional $\tau$-invariant subspace $S$ furnishes an $(n-1)$-dimensional subspace $V' \subset S$ such that $V' \oplus \tau(V') = S$, our strategy is to suitably select $S$ so that for a $2$-dimensional complement $T$ there is a vector $u \in T$ such that $\Span(u,\tau u) = T$. Then we can take $V = V' + \Span(u)$.

If there are two $1$-dimensional subspaces $C_{1,\lambda_1},\, C_{1,\lambda_2}$ spanned by $v_1,v_2$ respectively, we let
$S = \oplus_{(t,i) \neq (1,1), (1,2)} C_{t,\lambda_i}$ and $u = v_1+v_2$. So suppose that there is at most one eigenvalue, say $\lambda_1$, that possibly contributes $1$-dimensional subspaces $C_{t,\lambda_1}$'s. In that case, there exist $t',i'$ such that $d:=\dim_{\bar{k}} C_{t',\lambda_{i'}} >1$. Let $w_1,\dots,w_d$ be a Jordan basis for $ \C_{t',\lambda_{i'}}$. Define the $\tau$-invariant subspace $\tilde{C}_{t,\lambda_i} = \Span(w_1,\dots,w_{d-2})$, taken to be zero if $d=2$. Then we let $S = \left(\oplus_{(t,i) \neq (t',i')} C_{t,\lambda_i} \right) \oplus \tilde{C}_{t,\lambda_i}$ and 
 $u = w_d$.
\end{proof}

We take one step further by allowing $m \ge 2n$.
\begin{lem} \label{lem:m=2n+c}
Suppose $m \ge 2n$ and $\dim_{\bar{k}} E_{\tau,\lambda} \le m-n$ for every $\lambda \in \bar{k}$. Then there exists a $V \in \Gr_{\bar{k}}(n,m)$ such that $\dim_{\bar{k}} V \oplus \tau(V) =2n$.    
\end{lem}
\begin{proof}
The strategy is to find a $2n$-dimensional $\tau$-invariant subspace $S \subset \bar{k}^m$ for which $\dim_{\bar{k}} E_{\tau|_{S},\lambda_i} \le n$; then the claim will follow from Lemma \ref{lem:mu<=n}. We obtain $S$ by suitably truncating the $C_{t,\lambda_i}$'s. Set $\mu = \max_{i \in [s]} \dim_{\bar{k}} E_{\tau,\lambda_i}$. If $\mu=1$ then $\tau$ has $m$ distinct eigenvalues and we may take $S = \oplus_{i \in [2n]} C_{i,\lambda_i}$. Suppose that $1< \mu \le n$. Set $c=m-2n$. If there is some $C_{t',\lambda_{i'}}$ with $d = \dim_{\bar{k}} C_{t',\lambda_{i'}} \ge c$, let $w_1,\dots,w_d$ be a Jordan basis for $C_{t',\lambda_{i'}}$ and take $S =  \Big(\oplus_{(t,i) \neq (t',i')} C_{t,\lambda_i}\Big) \oplus \Span(w_1,\dots,w_{d-c})$. Otherwise, let $\beta>1$ be the smallest number of subspaces $C_{t_1,\lambda_{i_1}},\dots,C_{t_\beta,\lambda_{i_\beta}}$ for which $\dim_{\bar{k}} \oplus_{j \in [\beta]} C_{t_j,\lambda_{i_j}} = c+\ell$ for some $\ell \ge 0$. Then by the minimality of $\beta$ we must have that $\dim_{\bar{k}} C_{t_1,\lambda_{i_1}} \ge \ell$. Now replace $C_{t_1,\lambda_{i_1}}$ by an $\ell$-dimensional $\tau$-invariant subspace $\tilde{C}_{t_1,\lambda_{i_1}}$ obtained as the span of the first $\ell$ vectors of a Jordan basis of $C_{t_1,\lambda_{i_1}}$ and take $S = \Big(\oplus_{(t,i)\neq (t_j,\lambda_{i_j}), \, j \in [\beta]} C_{t,\lambda_i}\Big) \oplus \tilde{C}_{t_1,\lambda_{i_1}}$.

Next, suppose that $\mu > n$ and we may assume that $\dim_{\bar{k}} E_{\tau,\lambda_1}=\mu=n+c_1$ with $0<c_1 \le c$. We first treat the case $c_1=c$. In such a case $\dim_{\bar{k}} E_{\tau,\lambda_i} \le n$ for any $i >1$. Let $r$ be the number of $1$-dimensional $C_{t,\lambda_1}$'s, say $C_{1,\lambda_1},\dots,C_{r,\lambda_1}$. Then we must have that $$r+2(n+c-r) \le 2n+c \Leftrightarrow c \le r$$ and we can take $S = \Big(\oplus_{t=c+1}^{n+c} C_{t,\lambda_1}\Big) \oplus \Big( \oplus_{t; i>1} C_{t,\lambda_i}\Big)$. Next, suppose that $c_1 < c$. If $\dim_{\bar{k}} C_{t,\lambda_i}=1$ for every $t,i$, then there are $n+c-c_1$ $1$-dimensional $C_{t,\lambda_i}$'s associated to eigenvalues other than $\lambda_1$. In that case we can take $S$ to be the sum of $n$ subspaces associated to $\lambda_1$ and any other subspaces associated to eigenvalues different than $\lambda_1$. If on the other hand $\dim_{\bar{k}} C_{t',\lambda_{i'}}>1$ for some $t',i'$, then we replace $\bar{k}^m$ by $U_1$, the latter being the sum of all $C_{t,\lambda_i}$'s with the exception that $C_{t',\lambda_{i'}}$ has been replaced by a $\tilde{C}_{t',\lambda_{i'}} \subset C_{t',\lambda_{i'}}$ of dimension one less which we rename to $C_{t',\lambda_{i'}}$. Notice that this replacement does not change $\mu$. If $c-1=c_1$ or all $C_{t,\lambda_i}$'s in the decomposition of $U_1$ are $1$-dimensional, we are done by proceeding as above. If on the other hand $c-1> c_1$ and there is a $C_{t',\lambda_{i'}}$ of dimension larger than one, then replace $U_1$ by $U_2$, where the latter is the sum of all $C_{t,\lambda_{i}}$'s except the said $C_{t',\lambda_{i'}}$, which is replaced as above by a $C_{t',\lambda_{i'}}$ of dimension one less. Continuing inductively like this furnishes $S$. 
\end{proof}

We are now in a position to complete the proof of Proposition \ref{prp:tau}. Suppose first that $\dim_{\bar{k}} E_{\tau,1} \le m-n$. Then for $V \in \Gr_k(n,m)$ we have $\dim_k(V+\tau(V)) \le 2n$ with equality on an open set $\U \subset \Gr_k(n,m)$. With $A \in k^{m \times n}$ a basis of $V$, we have that $V \in \U$ if and only if some $2n \times 2n$ minor of the $m \times 2n$ matrix $[A \, \, \, T A]$ is non-zero. These minors are polynomials in the Pl\"ucker coordinates of $V$ with coefficients over $k$. By Lemma \ref{lem:m=2n+c} at least one of these polynomials has a non-zero valuation at some $V^* \in \Gr_{\bar{k}}(n,m)$ so that $\U$ is non-empty. Note that every $V \in \U$ necessarily does not intersect $\ker(\tau)$. Then for every $V \in \U$ we have $V \cap \tau(V) = 0$ so that the equality $\tau(v_1) = v_2$ implies $v_1=v_2=0$.  

Next, suppose that $\dim_{\bar{k}} E_{\tau,1} > m-n$. This implies $\dim_k E_{\tau,1} \ge n$. Let $S \in k^{m \times m}$ be an invertible matrix such that its first $n$ columns lie in $E_{\tau,1}$. Then the matrix $S^{-1} T S$ is block diagonal with the top left $n \times n$ block the identity matrix. Denote by $S_{[n]}^{-1}$ the top $n \times m$ block of $S^{-1}$ and let $\U$ be the non-empty open set of $\Gr_k(n,m)$ such that for every $V \in \U$ and every $A \in k^{m \times n}$ basis of $V$, the matrix $S_{[n]}^{-1} A$ is invertible. This makes $\U$ the complement of a hyperplane, since $\det(S_{[n]}^{-1} A)=0$ is a linear equation in the Pl\"ucker coordinates of $V$. Now let $V \in \U$ and suppose that $\tau(v_1) = v_2$ with $v_1, v_2 \in V$. Let $A$ be a basis of $V$ and write $v_i = A \xi_i$. Then the relation $TA \xi_1 = A \xi_2$ implies $S_{[n]}^{-1} A \xi_1 = S_{[n]}^{-1} A \xi_2$ and so $\xi_1 = \xi_2$.


\section{Proof of Theorem \ref{thm:permutations}}

\noindent We first need two lemmas.

\begin{lem} \label{lem:single-complete-cycle}
Let $\Pi$ be an $\ell \times \ell$ permutation matrix consisting of a single cycle, and let $\Sigma$ be an $\ell \times \ell$ diagonal matrix with its diagonal entries taking values in $\{1,-1\}$. Let $\Q$ be the ideal generated by the $2 \times 2$ determinants of the matrix $[\Sigma \Pi z \, \, \, z]$ over the $\ell$-dimensional polynomial ring $k[z] = k[z_1,\dots,z_{\ell}]$. Then $\height (\Q) = \ell -1$. 
\end{lem}
\begin{proof}
Note that the height of $\Q$ is the same as the height of the extension of $\Q$ in $\bar{k}[z]$, so we may assume that $k = \bar{k}$. Let $Y \subset \bar{k}^\ell$ be the vanishing locus of $\Q$. Clearly $v \in Y$ if and only if $v$ is an eigenvector of $\Sigma\Pi$. Hence $Y$ is the union of the eigenspaces of $\Sigma \Pi$, the latter being the irreducible components of $Y$. With $\sigma_i \in \{1,-1\}$ the $i$-th diagonal element of $\Sigma$, the eigenvalues of $\Sigma \Pi$ are the $\ell$ distinct roots of the equation $x^{\ell}=\sigma_1 \cdots \sigma_{\ell}$. Hence $\Sigma\Pi$ is diagonalizable with $\ell$ distinct eigenvalues, i.e., each eigenspace has dimension $1$. Thus $Y$ has pure dimension $1 = \dim Y = \dim \bar{k}[z]/I_{\Sigma \Pi}$ whence $\height(\Q) = 1$.
\end{proof}

\begin{lem} \label{lem:incomplete-cyles-codimension}
Let $\Pi$ be an $m \times m$ permutation matrix consisting of $c$ cycles and $\Sigma$ an $m \times m$ diagonal matrix with diagonal entries taking values in $\{1,-1\}$. For every $i \in [c]$ let $\I_i \subset [m]$ be the indices that are cycled by cycle $i$. Let $\bar{\I} \subset [m]$ such that $\#\bar{\I} \ge 2$ and $\I_i \not\subset \bar{\I}$ for every $i \in [c]$. Let $\Q$ be the ideal generated by the $2 \times 2$ determinants of the row-submatrix $\Phi$ of $[x \, \, \, \Sigma \Pi x]$ indexed by $\bar{\I}$. Viewing $\Q$ as an ideal of the polynomial ring over $k$ in the indeterminates that appear in $\Phi$, we have that $\height(\Q) = \#\bar{\I}-1$. 
\end{lem}
\begin{proof}
Let $\Phi=[x \, \, \, \Sigma \Pi x]_{\bar{\I}}$ be the said submatrix. Let $r \in [c]$ be such that $\bar{\I} \cap \I_r \neq \emptyset$. Since $\I_r \not\subset \bar{\I}$, we can partition $\bar{\I} \cap \I_r$ into subsets $\bar{\I}_{r j}$ for $j \in [s_r]$ for some $s_r$, such that each $\Phi_{r j} = [x \, \, \, \Sigma \Pi x]_{\bar{\I}_{r j}}$ has up to a permutation of the rows the form 
\begin{align}
\Phi_{rj} = 
\begin{bmatrix}
x_{\alpha} & \sigma_{\beta} x_{\beta} \\
x_{\alpha+1} & \sigma_{\alpha}x_{\alpha} \\
\vdots & \vdots \\
x_{\alpha+\ell-2} & \sigma_{\alpha+\ell-3} x_{\alpha+\ell-3} \\
x_{\gamma} & \sigma_{\alpha+\ell-2} x_{\alpha+\ell-2}
\end{bmatrix}  \nonumber
\end{align} Here $\sigma_i \in \{1,-1\}$ and $x_{\alpha},\dots,x_{\alpha+\ell-2},x_{\beta},x_{\gamma}$ are distinct variables appearing only in $\Phi_{rj}$. Note that there is a total of $s=\sum_{\bar{\I}_r \neq \emptyset} s_r$ blocks $\Phi_{rj}$ and a total of $\#\bar{\I}+s$ distinct indeterminates appearing in $\Phi$. Let $T$ be the general determinantal ring over $k$ of $2 \times 2$ determinants of a $\#\bar{\I} \times 2$ matrix of variables. Then it is very well known that $T$ is Cohen-Macaulay of dimension $\#\bar{\I}+1$ \cite{BrunsVetter:1988}. Taking quotient with $\#\bar{\I}-s$ suitable linear forms we obtain the quotient ring associated to $\Q$. Taking quotient with extra $s$ linear forms we can obtain the quotient ring of an ideal of the form appearing in Lemma \ref{lem:single-complete-cycle}. Then as per Lemma \ref{lem:single-complete-cycle} this is $1$-dimensional so that the total sequence of $\#\bar{\I}$ linear forms is a $T$-regular sequence. 
\end{proof}

\begin{remark}
Ignoring the sign matrix $\Sigma$, a geometric view of the proof of Lemma \ref{lem:incomplete-cyles-codimension} is the following. When $k=\bar{k}$ the ideal $\Q$ corresponds to a rational normal scroll of dimension $s+1$.  Then we take a sequence of $s$ hyperplane sections of that scroll, each time getting a new scroll of dimension one less until the scroll degenerates to the union of eigenspaces of a cyclic permutation. See \cite{ConcaFaenzi:2017} for a complete classification of rational normal scrolls that arise as hyperplane sections of rational normal scrolls, see also
\cite{CatalanoJohnson:1997} for the free resolution of ideals of $2 \times 2$ determinants of a matrix of linear forms with two columns. 
\end{remark} 

It is enough to bound as claimed the dimension of $U_{\rho_2 \rho_1 \pi,\rho_2}$ where $\pi$ is some permutation. Since the dimension of locally finite type $k$-schemes is preserved under any field extension (exercise 11.2.J in \cite{Vakil-AG}) we may assume that $k = \bar{k}$. Let $R_1, R_2, \Pi$ be matrix representations of $\rho_1,\rho_2, \pi$ on the canonical basis of $k^m$. For a closed point $v \in U_{\rho_2 \rho_1 \pi,\rho_2}$ we have $R_2 R_1 \Pi v = \lambda R_2 v$ for some $\lambda \neq 0,1$. For $i=1,2$, let $I_i \subset [m]$ be the indices that correspond to $\im(R_i)$, and similarly $K_i$ the indices that correspond to $\ker(R_i)$. If $i \in I_2 \cap K_1$, then it is clear that $v_i$ must be zero, because $\lambda \neq 0$. If $\pi(i) \in I_2 \cap K_1$, then we must also have $v_{\pi(i)}=0$ for the same reason. If $\pi(i) \in I_2 \cap I_1$, then again $v_{\pi(i)}=0$ because we already have $v_i = 0$ and $\lambda \neq 0$. This \emph{domino effect} either forces $v$ to be zero in the entire orbit of $i$, or until an index $j$ in the orbit of $i$ is reached such that $\pi(j) \in K_2 \cap K_1$. Let $I_{\text{domino}} \subset I_2 $ be the coordinates of $v$ that are forced to zero by the union of the domino effects for every $i \in I_2 \cap K_1$. Clearly $ I_2 \setminus I_{\text{domino}} \subset I_2 \cap I_1$. Let $i \in I_2 \setminus I_{\text{domino}}$; if it so happens that $\pi(i) = i$, then we must have that $v_i = 0$ because $\lambda \neq 1$. Consequently the coordinates of $v$ that correspond to fixed points of $\pi$ and lie in $I_2 \setminus I_{\text{domino}}$ must be zero. Letting $I_{\text{fixed}} \subset I_2 \setminus I_{\text{domino}}$ be the set containing these indices, $v$ must lie in the linear variety defined by the vanishing of the coordinates indexed by $I_{\text{domino}} \cup I_{\text{fixed}}$. 

Next, let $\bar{\pi}_1,\dots,\bar{\pi}_{c'}$ be all the $c' \ge 0$ cycles of $\pi$ of length at least two that lie entirely in $I_2 \setminus (I_{\text{domino}} \cup I_{\text{fixed}})$. Let $C_{i} \subset [m]$ be the indices cycled by $\bar{\pi}_{i}$. Since $\lambda \neq 0$, it is clear that $v_{C_i}$ must be an eigenvector of $\bar{\pi}_i$, and so by Lemma \ref{lem:single-complete-cycle} $v_{C_i}$ must lie in a codimension-$(\#C_i-1)$ variety. Adding codimensions over $i \in [c']$, and letting $I_{\text{cycles}} = \bigcup_{i \in [c']} C_i$, we get that $v_{I_{\text{cycles}}}$ must lie in a variety of codimension $\sum_{i \in [c']}(\#C_i-1)$. Moreover, we may assume that the set $I_{\text{incomplete}} = I_2 \setminus (I_{\text{domino}} \cup I_{\text{fixed}} \cup I_{\text{cycles}})$ does not contain any complete cycles, and if $I_{\text{incomplete}} \neq \emptyset$ Lemma \ref{lem:incomplete-cyles-codimension} gives that $v_{I_{\text{incomplete}}}$ must lie in a codimension-$(\#I_{\text{incomplete}}-1)$ variety. 

Let $\Y_{\text{domino}}, \Y_{\text{fixed}}, \Y_{\text{cycles}}, \Y_{\text{incomplete}}$ be the varieties defined by the vanishing of the coordinates in $I_{\text{domino}}$, the vanishing of the coordinates in $I_{\text{fixed}}$, as well as the vanishing
of the $2 \times 2$ determinants of the matrix $[x \, \, \, \Pi x]$ indexed by 
$I_{\text{cycles}}$ and $I_{\text{incomplete}}$ respectively. Noting that these varieties are all associated with disjoint polynomial rings and that $\# I_{\text{domino}}+\#I_{\text{fixed}} +\#I_{\text{cycles}}+\#I_{\text{incomplete}}=\#I_2$, the above analysis gives that $v$ must lie in a variety $\Y = \Y_{\text{domino}} \times \Y_{\text{fixed}} \times \Y_{\text{cycles}} \times \Y_{\text{incomplete}}$ so that
\begin{align}
\codim \Y &\ge \# I_{\text{domino}} + \#I_{\text{fixed}} + \sum_{i \in [c']} (\#C_i-1) + \max\{\# I_{\text{incomplete}}-1,0\} \nonumber \\
&= \#I_2-c' - \# I_{\text{incomplete}} + \max\{\# I_{\text{incomplete}}-1,0\} \nonumber. 
\end{align} If $I_{\text{incomplete}}=\emptyset$, then $\codim \Y \ge  \# I_2 - c'$. Since $c' \le \# I_2 / 2$, we have that $\codim \Y \ge  \# I_2 / 2 \ge \floor*{\# I_2/2}$. If on the other hand $I_{\text{incomplete}} \neq \emptyset$, then $c' \le \floor*{ (\#I_2-1) / 2}$, so that $\codim \Y \ge \#I_2 - \floor*{ (\#I_2-1) / 2}- 1 \ge \floor*{\# I_2/2}$, with the last inequality separately verified for $\#I_2$ odd or even. 

\section{Proof of Corollary \ref{cor:signed-permutations}}
If $\rho_1 \in \sR_n$ and $\rho_2 \in \sR_{2n}$, then the claim is a direct corollary of Theorems \ref{thm:tau1-tau2} and \ref{thm:permutations}. Otherwise, a similar set of arguments as in the proof of Theorem \ref{thm:permutations} establishes that
$\dim \U_{\rho_2 \rho_1 \pi_1, \rho_2 \pi_2}^{\pm} \le m-\floor*{\rank(\rho_2)/2}$, where now
$\U_{\rho_2 \rho_1 \pi_1, \rho_2 \pi_2}^{\pm} = \U_{\rho_2 \rho_1 \pi_1, \rho_2 \pi_2} \setminus \ker(\rho \tau_1 + \tau_2)$. Moreover, an identical argument as in the end of the proof of Proposition \ref{prp:tau} shows that we can adjust that proposition as follows: ``Suppose 
$\dim_{\bar{k}} E_{\tau,\lambda} \le m-n$ for every $\lambda \neq 1,-1$. Then for a general $n$-dimensional subspace $V$ and $v_1,v_2 \in V$ we have $\tau(v_1) = v_2$ only if $v_1 = v_2$ or $v_1=-v_2$." Combining everything together establishes the claim.

\section{Proof of Proposition \ref{prp:generalHS}}

Let $A \in k^{m \times n}$ be a basis of $V$. If $\tau_1(v_1) = \tau_2(v_2)$ then $\tau_2(v_2) \in \tau_1(V)$ and so $\rank ([T_1A \, \, \, T_2 A \xi]) \le n$ for $\xi \in k^n$ with $v_2=A\xi$. We show that for general $V,\xi$ this can not happen unless $\tau_1=\tau_2$, in which case $v_1-v_2 \in \ker(\tau_1)$ and so $v_1 = v_2$. Suppose $\tau_1 \neq \tau_2$. We show the existence of $A,\xi$ such that $\rank ([T_1A \, \, \, T_2 A \xi]) = n+1$. Since $\tau_1 \neq \lambda \tau_2$ for all $\lambda \in k$, there exists some $v \in k^m$ such that $\tau_1(v),\tau_2(v)$ are linearly independent. Let $W = \Span\big(\tau_1(v),\tau_2(v)\big)$. Since $\rank(\tau_1) \ge n+1$, any complement $C$ of $W \cap \im(\tau_1)$ in $\im(\tau_1)$ has dimension at least $n-1$. Let $C_1$ be a subspace of $C$ of dimension $n-1$. Let $V_1$ be a subspace of $\tau_1^{-1}(C_1)$ of dimension $n-1$ such that $C_1 = \tau_1(V_1)$. Then for $V = V_1 + \Span(v)$ we have $\dim\big(\tau_1(V) + \tau_2(v)\big)=n+1$.

\bibliographystyle{humannat}
\bibliography{DCHS-arxiv-16Dec20}

\begin{thebibliography}{}

\bibitem[\protect\astroncite{Bruns and Vetter}{1988}]{BrunsVetter:1988}
Bruns, W. and U.~Vetter\leavevmode\nopagebreak\newline 1988.
\newblock {\em Determinantal Rings}.
\newblock Springer, Berlin/Heidelberg.

\bibitem[\protect\astroncite{Catalano-Johnson}{1997}]{CatalanoJohnson:1997}
Catalano-Johnson, M.~L.\leavevmode\nopagebreak\newline 1997.
\newblock The resolution of the ideal of 2x2 minors of a 2xn matrix of linear
  forms.
\newblock {\em Journal of Algebra}, 187:39--48.

\bibitem[\protect\astroncite{Conca and Faenzi}{2017}]{ConcaFaenzi:2017}
Conca, A. and D.~Faenzi\leavevmode\nopagebreak\newline 2017.
\newblock A remark on hyperplane sections of rational normal scrolls.
\newblock {\em Bull. Math. Soc. Sci. Math. Roumanie}, 60 (108)(4):337--349.

\bibitem[\protect\astroncite{Dokmani{\'c}}{2019}]{dokmanic2019permutations}
Dokmani{\'c}, I.\leavevmode\nopagebreak\newline 2019.
\newblock Permutations unlabeled beyond sampling unknown.
\newblock {\em IEEE Signal Processing Letters}, 26(6):823--827.

\bibitem[\protect\astroncite{Han et~al.}{2018}]{han2018recovery}
Han, D., F.~Lv, and W.~Sun\leavevmode\nopagebreak\newline 2018.
\newblock Recovery of signals from unordered partial frame coefficients.
\newblock {\em Applied and Computational Harmonic Analysis}, 44(1):38--58.

\bibitem[\protect\astroncite{Lv and Sun}{2018}]{lv2018real}
Lv, F. and W.~Sun\leavevmode\nopagebreak\newline 2018.
\newblock Real phase retrieval from unordered partial frame coefficients.
\newblock {\em Advances in Computational Mathematics}, 44(3):879--896.

\bibitem[\protect\astroncite{Roman}{2008}]{Roman}
Roman, S.\leavevmode\nopagebreak\newline 2008.
\newblock {\em Advanced Linear Algebra}.
\newblock New York: Springer, third edition.

\bibitem[\protect\astroncite{Tsakiris and Peng}{2019}]{tsakiris2019homomorphic}
Tsakiris, M. and L.~Peng\leavevmode\nopagebreak\newline 2019.
\newblock Homomorphic sensing.
\newblock In {\em International Conference on Machine Learning}, Pp.~
  6335--6344.

\bibitem[\protect\astroncite{Unnikrishnan
  et~al.}{2018}]{unnikrishnan2018unlabeled}
Unnikrishnan, J., S.~Haghighatshoar, and
  M.~Vetterli\leavevmode\nopagebreak\newline 2018.
\newblock Unlabeled sensing with random linear measurements.
\newblock {\em IEEE Transactions on Information Theory}, 64(5):3237--3253.

\bibitem[\protect\astroncite{Vakil}{2017}]{Vakil-AG}
Vakil, R.\leavevmode\nopagebreak\newline 2017.
\newblock {\em The Rising Sea: Foundations of Algebraic Geometry}.

\end{thebibliography}

\end{document}